\documentclass[12pt]{article}
\usepackage{color}
\usepackage{graphicx}
\usepackage{subcaption} 
\usepackage{algorithm}
\usepackage{algpseudocode} 
\usepackage{mathtools} 
\usepackage{makecell}
\usepackage{booktabs}

\usepackage{setspace}
\usepackage{amssymb,amsmath,acronym,amsfonts}
\usepackage[square, comma, sort&compress, numbers]{natbib}
\usepackage{enumerate}
\usepackage{epstopdf}

\newtheorem{corollary}{Corollary}[section]
\newtheorem{theorem}{Theorem}[section]
\newtheorem{lemma}{Lemma}[section]
\newtheorem{definition}{Definition}[section]
\newtheorem{proposition}{Proposition}[section]
\newtheorem{example}{Example}[section]
\newtheorem{assum}{Assumption}[section]
\newtheorem{algo}{Algorithm}[section]
\newtheorem{Remark}{Remark}[section]

\numberwithin{equation}{section}
\def\x{{\bf x}}
\def\y{{\bf y}}

\def\w{{\bf w}}

\def\bc{\begin{corl}}
\def\bc{\end{corl}}
\def\ba{\begin{algo}}
\def\ea{\end{algo}}
\def\br{\begin{Remark}}
\def\er{\end{Remark}}
\def\bs{\begin{assum}}
\def\es{\end{assum}}
\def\bt{\begin{theorem}}
\def\et{\end{theorem}\vskip 3pt}
\def\bl{\begin{lemma}}
\def\el{\end{lemma}}
\def\ep{\end{proposition}}
\def\bp{\begin{proposition}}
\def\qed{\hfill{$\Box$}\vskip 5pt}
\def\be{\begin{example}}
\def\ee{\end{example}}
\def\bd{\begin{definition}}
\def\ed{\end{definition}}
\def\bc{\begin{corollary}}
\def\ec{\end{corollary}}
\def\proof{\noindent\it Proof. \hspace{1mm}\rm}
\input epsf
\begin{document}
\title{\bf $\mathrm M$-Eigenpairs of Partially Symmetric Tensors: Exact Reformulation and Perturbation Bounds}
\date{}
\author{Zhuolin Du, Hanxin Liu, Yisheng Song\thanks{Corresponding author E-mail: yisheng.song@cqnu.edu.cn}\\
	School of Mathematical Sciences,  Chongqing Normal University, \\
	Chongqing, 401331, P.R. China. \\ Email: duzhuolin728@163.com (Du); 2531417503@qq.com (Liu); \\ yisheng.song@cqnu.edu.cn (Song)}
\maketitle
\vspace{-0.6cm}
\begin{abstract}
In this paper, we consider the computation of $M$-eigenpairs of fourth order partially symmetric tensors arising from elasticity theory. First, a lifted fourth order tensor is constructed, and the original $M$-eigenvalue problem is reformulated as a parameterized
generalized tensor eigenvalue problem under the $\mathbf B_{\alpha,\beta}$-normalization.
Then, an exact correspondence between the two eigenvalue problems is established, which provides a procedure for computing all real $M$-eigenpairs through the proposed reformulation.
Furthermore, perturbation bounds for the largest $M$-eigenvalue are derived, and the lifted reformulation is shown to preserve these bounds without introducing any additional relaxation. Finally, numerical experiments are reported to show the effectiveness of the proposed method.


\smallskip


\noindent{\bf AMS Subject Classification(2010):} 

\end{abstract}

\newpage
\section{Introduction}
Computing eigenvalues and eigenvectors of higher-order tensors has become a fundamental topic in numerical multilinear algebra \cite{Qi05,KB09,Zhao23,PA24} and polynomial optimization \cite{Lim05,QC18,ZL24}. Various notions of tensor eigenvalues have been introduced and extensively studied, including $H$-, $E$-, $Z$-, $C$-, $D$-, and $M$-eigenvalues, as well as generalized tensor eigenvalues formulations \cite{QC18}. These notions play an important role in areas such as elasticity theory \cite{QDH09,Zhao23}, quantum entanglement \cite{WQZ09, LLL19}, and hypergraph theory \cite{CCQ16}, thereby motivating extensive research on the theoretical analysis and numerical computation of tensor eigenpairs.

Among these notions, $M$-eigenvalues provide a fundamental tool for studying fourth order partially symmetric tensors. Such tensors arise naturally in continuum mechanics as elasticity tensors \cite{QDH09,Zhao23,LCL22}. Their spectral properties are closely related to the analysis of definiteness, strong ellipticity, and material stability in nonlinear elasticity. Consequently, the analysis and computation of $M$-eigenpairs have attracted sustained attention in recent years \cite{WQZ09,WCW23,ZL24}.

\bd\label{def1.1}\cite{WQZ09} Let $\mathcal A=(a_{ijkl})\in \mathbb R^{m\times n\times m\times n}$ be a fourth order tensor. If
\[
a_{ijkl}=a_{kjil}=a_{ilkj}=a_{klij}
\]
for all $i,k\in[m]$ and $j,l\in[n]$, where $[m] = \{1, 2, \dots, m\}$, then $\mathcal A$ is called a fourth order partially symmetric tensor.
\ed

For simplicity, let $\mathbb{P}^{m\times n\times m\times n}$ denote the set of all real fourth order partially symmetric tensors.

\bd\label{def1.2}
Let $\mathcal{A}=(a_{ijkl})\in\mathbb{P}^{m\times n\times m\times n}$. An $M$-eigenpair of $\mathcal{A}$ is a triple $(\lambda,\x,\y)\in\mathbb{R}\times(\mathbb{R}^m\setminus\{\mathbf{0}\})\times(\mathbb{R}^n\setminus\{\mathbf{0}\})$ satisfying
\begin{equation}\label{1.1}
\mathcal{A}\cdot \y\x\y=\lambda\x,\qquad
\mathcal{A}\x\y\x\cdot=\lambda\y,\qquad
\x^\top\x=1,\qquad
\y^\top\y=1,
\end{equation}
where
 $$(\mathcal{A}\cdot\y\x\y)_i=\sum_{\substack{k\in [m]; ~j, l \in [n]}}a_{ijkl}y_j x_ky_l
\quad \mbox{and}\quad(\mathcal{A}\x\y\x\cdot)_l = \sum_{\substack{i,k\in [m];~j \in [n]}} a_{ijkl} x_i y_j x_k,$$
then $\lambda$ is called an $M$-eigenvalue of $\mathcal{A}$, and $\x$ and $\y$ are called the associated left and right $M$-eigenvectors, respectively.
\ed

The tensor
$\mathcal A\in\mathbb P^{m\times n\times m\times n}$ induces the biquadratic form
$$
\mathcal A\mathbf x\mathbf y\mathbf x\mathbf y
:=
\sum_{\substack{i,k\in[m]\\ j,l\in[n]}}
a_{ijkl}x_i y_j x_k y_l
=
\mathbf x^\top
(\mathcal A\cdot\mathbf y\mathbf x\mathbf y).
$$
For later use, denote the largest and smallest $M$-eigenvalues
of $\mathcal A$, respectively, by
\begin{equation}\label{1.2}
\lambda_{\max}^{M}(\mathcal A)
:=\max_{\substack{\mathbf x^\top\mathbf x=1\\
                 \mathbf y^\top\mathbf y=1}}
\mathcal A\mathbf x\mathbf y\mathbf x\mathbf y,
\qquad
\lambda_{\min}^{M}(\mathcal A):=
\min_{\substack{\mathbf x^\top\mathbf x=1\\
                 \mathbf y^\top\mathbf y=1}}
\mathcal A\mathbf x\mathbf y\mathbf x\mathbf y.
\end{equation}

Despite their importance, the computation of $M$-eigenvalues remains challenging. As shown in Definition \ref{def1.2}, an $M$-eigenpair is characterized by a coupled nonlinear system involving two vector variables subject to two independent spherical constraints. Moreover, the variational characterizations in \eqref{1.2} show that computing the extremal $M$-eigenvalues amounts to solving biquadratic optimization problems over a product of unit spheres. These problems are generally nonconvex and computationally difficult. Therefore, existing studies have mainly focused on inclusion bounds, sufficient conditions for strong ellipticity, and numerical methods for extremal $M$-eigenvalues.

A considerable amount of work has been devoted to $M$-eigenvalues of fourth order partially symmetric tensors, especially in connection with strong ellipticity. Since strong ellipticity is equivalent to the positivity of the smallest $M$-eigenvalue, many studies have focused on constructing $M$-eigenvalue inclusion intervals and computable bounds. For instance, Li et al.\ \cite{LLL19} derived $M$-eigenvalue inclusion intervals for fourth order partially symmetric tensors, while Che et al.\ \cite{CCW20} established bounds for $M$-eigenvalues and the $M$-spectral radius of nonnegative fourth order partially symmetric tensors. Related advances include refined localization sets and $S$-type inclusion theorems; see, for example, \cite{CCW20,CCX20,CCZ21,HLW20,HLX21,HXL20,LCL22,LL21}. These results provide valuable criteria for estimating $M$-eigenvalues. However, since most of these conditions are sufficient rather than necessary, they cannot completely characterize strong ellipticity. This motivates the development of numerical approaches for directly computing $M$-eigenpairs.

Alongside these theoretical developments, several numerical methods have been proposed for computing extremal $M$-eigenvalues. The largest $M$-eigenvalue can be formulated as a biquadratic optimization problem over unit spheres, and Wang et al.\ \cite{WQZ09} used this formulation to develop a practical power type method. Building on related optimization viewpoints, later studies proposed alternating and gradient type methods for extremal $M$-eigenvalue problems. For instance, an alternating shifted power method was developed for fourth order partially symmetric tensors \cite{WCW23}, while Chen et al.\ \cite{CH22} introduced an efficient alternating minimization method for a class of fourth degree polynomial optimization problems applicable to the smallest $M$-eigenvalue. Du et al.\ \cite{DW24} further proposed a generalized inverse power method for the smallest $V$-singular value of partially symmetric tensors, which reduces to the smallest $M$-eigenvalue when $(p,q)=(2,2)$. More recently, Du et al.\ \cite{DS26} proposed a memory gradient method for computing extreme $M$-eigenvalues of fourth order partially symmetric tensors. These developments have substantially advanced the computation of extremal spectral quantities. Nevertheless, most existing methods focus on the largest, smallest, or other extremal $M$-eigenvalues, whereas the computation of the complete set of real $M$-eigenpairs has received much less attention.

Beyond the computation of extremal $M$-eigenvalues, Zhao \cite{Zhao23} recently proposed a lifting strategy for computing all $M$-eigenpairs of partially symmetric tensors. By constructing an associated lifting tensor, the $M$-eigenvalue problem was transformed into a standard $Z$-eigenvalue problem, and an explicit correspondence between $M$-eigenpairs and $Z$-eigenpairs was established. Consequently, all real $M$-eigenpairs can be recovered through the lifted tensor formulation. However, this approach relies on the standard Euclidean normalization of $Z$-eigenvalues. A natural question is whether a more general normalization framework can be introduced while preserving the exact correspondence and spectral information.

Motivated by the above observations, this paper develops a parameterized lifting framework for computing $M$-eigenpairs of fourth order partially symmetric tensors. By introducing the weighted normalization matrix $\mathbf B_{\alpha,\beta}$, the original $M$-eigenvalue problem is reformulated as a symmetric generalized tensor eigenvalue problem. The main contributions of this paper are summarized as follows:

(i) A lifted tensor construction with a parameterized $\mathbf B_{\alpha,\beta}$-normalization is proposed, which extends the standard $Z$-eigenvalue reformulation.

(ii) An exact correspondence between $M$-eigenpairs of the original tensor and $\mathbf B_{\alpha,\beta}$-eigenpairs of the lifted tensor is established, together with an explicit recovery procedure for all real $M$-eigenpairs.

(iii) Perturbation bounds for the largest $M$-eigenvalue are derived in both the original and lifted formulations. It is shown that the lifting transformation preserves the perturbation estimates up to a scaling factor.

The remainder of this paper is organized as follows. Section 2 presents the lifted tensor construction and establishes the correspondence between $M$-eigenpairs and $\mathbf B_{\alpha,\beta}$-eigenpairs. Section 3 develops perturbation bounds for the largest $M$-eigenvalue and analyzes their preservation under the lifting transformation. Numerical experiments are reported in Section 4, and conclusions are drawn in Section 5.

\setcounter{equation}{0}
\section{A parameterized lifted tensor eigenvalue reformulation}
In this section, we reformulate the $M$-eigenvalue problem for a fourth order partially symmetric tensor as a generalized tensor eigenvalue problem for an associated symmetric lifted tensor. The reformulation is based on a parameterized weighted normalization matrix $\mathbf B_{\alpha,\beta}$.

For positive integers $p$ and $N$, let
\[
\mathbb R^{[p,N]}
:=
\mathbb R^{\overbrace{N\times\cdots\times N}^{p\ {\rm times}}}
\]
denote the space of real $p$th-order $N$-dimensional tensors. We first recall the definitions of $Z$-eigenpairs and mode-$k$ $\mathcal B$-eigenpairs needed below.

\bd\label{defZ}\cite{Qi05}
Let $\mathcal A\in\mathbb R^{[p,N]}$ be a symmetric tensor. A pair $(\lambda,\x)\in
\mathbb R\times(\mathbb R^N\setminus\{\mathbf0\})$ is called a $Z$-eigenpair of $\mathcal A$ if
\[
\mathcal A\x^{p-1}=\lambda\x,
\qquad
\x^\top\x=1 .
\]
The scalar $\lambda$ is called a $Z$-eigenvalue and $\x$ is called the associated $Z$-eigenvector.
\ed

For a tensor \(\mathcal A=(a_{i_1\cdots i_m})\in\mathbb R^{[p,N]}\) and a vector \(\x\in\mathbb R^N\), denote by \(\mathcal A^{(k)}\x^{p-1}\in\mathbb R^N\) the vector whose \(i_k\)-th component is
\[
(\mathcal A^{(k)}\x^{p-1})_{i_k}
=
\sum_{i_1,\ldots,i_{k-1},i_{k+1},\ldots,i_p=1}^{N}
a_{i_1\cdots i_p}
x_{i_1}\cdots x_{i_{k-1}}x_{i_{k+1}}\cdots x_{i_p}.
\]

\bd \label{df2.1}\cite{CHZ16}
Let $\mathcal{A} \in \mathbb{R}^{[p,N]}$ and $\mathcal{B} \in \mathbb{R}^{[q,N]}$.
Assume that the homogeneous polynomial $\mathcal{B}\x^{q}$ is not identically zero. For $1 \leq k \leq p$, if there exist $\lambda \in \mathbb{R}$ and $\x \in \mathbb{R}^N\setminus\{\bf 0\}$ such that

(i) when $p = q$,
\[
\mathcal{A}^{(k)} \x^{p-1} = \lambda \mathcal{B} \x^{p-1},
\]

(ii) when $p \neq q$,
\[
\mathcal{A}^{(k)} \x^{p-1} = \lambda \mathcal{B} \x^{q-1}, \quad \mathcal{B} \x^{q} = 1,
\]
then $\lambda$ is called a mode-$k$ $\mathcal{B}$-eigenvalue of $\mathcal{A}$, $\x$ is called a mode-$k$ $\mathcal{B}$-eigenvector associated with $\lambda$, and $(\lambda,\x)$ is called a mode-$k$ $\mathcal{B}$-eigenpair of $\mathcal{A}$.


\ed

With this generalized definition in hand, we specialize the generalized eigenvalue framework of Definition~\ref{df2.1} to the $M$-eigenpair problem associated with $\mathcal A=(a_{ijkl})\in
\mathbb P^{m\times n\times m\times n}$. The key idea is to embed the two coupled equations in in \eqref{1.1} into a single  mode-$1$ tensor equation by introducing an lifted tensor.

The first step is to construct an associated lifted tensor in $\mathbb{R}^{[4,m+n]}$.
Define $\mathcal{B}_{\mathcal{A}} = (b_{t_1t_2t_3t_4})\in \mathbb{R}^{[4,m+n]}$ by
\begin{equation}\label{2.1}
b_{t_1 t_2 t_3t_4}=
\begin{cases}
a_{t_1t_2-mt_3t_4-m}, & \text{if } t_1,t_3\in [m],\ t_2,t_4\in m+[n],\\
a_{t_2t_3-mt_4t_1-m}, & \text{if } t_1,t_3\in m+[n],\ t_2,t_4\in [m],\\
0, & \text{otherwise},
\end{cases}
\end{equation}
where $m+[n] = \{m+1, m+2, \dots, m+n\}$. Let $\w=(\x^\top,\y^\top)^\top \in \mathbb{R}^{m+n}$, then
\begin{equation}\label{2.2}
\mathcal{B}_{\mathcal{A}} \w^3=
\begin{pmatrix}
\mathcal{A}\cdot\y\x\y \\
\mathcal{A}\x\y\x\cdot
\end{pmatrix}.
\end{equation}
Indeed, if $t_1\in[m]$, then
$$
\begin{aligned}
(\mathcal{B}_{\mathcal{A}}\w^3)_{t_1} &= \sum_{t_2,t_3,t_4 \in [m+n]} b_{t_1t_2t_3t_4} w_{t_2} w_{t_3}w_{t_4} \\
&= \sum_{t_3 \in [m],\, t_2,t_4 \in m+[n]} a_{t_1t_2-mt_3t_4-m} w_{t_2} w_{t_3}w_{t_4}
\end{aligned}
$$
$$
\begin{aligned}
&= \sum_{t_3 \in [m],\, t_2-m,t_4-m\in [n]} a_{t_1t_2-mt_3t_4-m} y_{t_2-m} x_{t_3}y_{t_4-m} \\
&= \sum_{k \in [m], j,l\in [n]} a_{t_1jkl} y_jx_ky_l=(\mathcal{A}\cdot\y\x\y)_{t_1}.
\end{aligned}
$$
Similarly, if $t_1 \in m+[n]$, then
$$
\begin{aligned}
(\mathcal{B}_{\mathcal{A}}\w^3)_{t_1} &= \sum_{t_2,t_3,t_4 \in [m+n]} b_{t_1t_2t_3t_4} w_{t_2}w_{t_3}w_{t_4}\\
&= \sum_{t_2,t_4 \in [m],\, t_3\in m+[n]} a_{t_2t_3-mt_4t_1-m} w_{t_2} w_{t_3} w_{t_4}\\
&= \sum_{t_2,t_4 \in [m],\, t_3-m\in [n]} a_{t_2t_3-mt_4t_1-m}x_{t_2} y_{t_3-m}x_{t_4} \\
&= \sum_{i,k \in [m], \,j\in [n]} a_{ijkt_1-m} x_i y_jx_k = (\mathcal{A}\x\y\x\cdot)_{t_1-m}.
\end{aligned}
$$
Consequently, \eqref{2.2} holds.

In addition, by \eqref{2.1} and the fact that $\mathcal A$ is partially symmetric, the lifted tensor $\mathcal B_{\mathcal A}$ satisfies
$$
b_{t_1 t_2 t_3 t_4} = b_{t_3 t_2 t_1 t_4} = b_{t_1 t_4 t_3 t_2} = b_{t_3 t_4 t_1 t_2}.
$$

We next symmetrize $\mathcal B_{\mathcal A}$. Define  $\bar{\mathcal B}_{\mathcal A}=(\bar b_{t_1t_2t_3t_4})$ by
\begin{equation}\label{2.3}
\bar{b}_{t_1 t_2 t_3 t_4} = \frac{1}{6} \left(b_{t_1 t_2 t_3 t_4} + b_{t_1 t_3 t_2 t_4} + b_{t_1 t_2 t_4 t_3}+b_{t_2t_1t_3t_4}+b_{t_2t_1t_4t_3}+b_{t_3t_1t_4t_2}\right).
\end{equation}
By construction, $\bar{\mathcal B}_{\mathcal A}$ satisfies all permutations of its indices and hence is symmetric.

Then, for each $t_1\in[m+n]$,
$$
\begin{aligned}
(\bar{\mathcal{B}_\mathcal{A}} \w^3)_{t_1} &= \sum_{t_2,t_3,t_4 \in [m+n]} \bar{b}_{t_1 t_2 t_3 t_4} w_{i_2} w_{i_3} w_{i_4} \\
&= \frac{1}{6} \sum_{t_2,t_3,t_4 \in [m+n]} \left(b_{t_1 t_2 t_3 t_4} + b_{t_1 t_3 t_2 t_4} + b_{t_1 t_2 t_4 t_3}+b_{t_2t_1t_3t_4}+b_{t_2t_1t_4t_3}+b_{t_3t_1t_4t_2}\right) w_{t_2} w_{t_3} w_{t_4} \\
&= (\mathcal{B}_\mathcal{A} \w^3)_{t_1}.
\end{aligned}
$$
Similarly,
$$
\begin{aligned}
\bar{\mathcal{B}_\mathcal{A}} \w^4 &= \sum_{t_1,\dots,t_4 \in [m+n]} \bar{b}_{t_1 t_2 t_3 t_4} w_{t_1} w_{t_2} w_{t_3} w_{t_4} \\
&= \frac{1}{6} \sum_{t_1,\dots,t_4 \in [m+n]} \left(b_{t_1 t_2 t_3 t_4} + b_{t_1 t_3 t_2 t_4} + b_{t_1 t_2 t_4 t_3}+b_{t_2t_1t_3t_4}+b_{t_2t_1t_4t_3}+b_{t_3t_1t_4t_2}\right) w_{t_1} w_{t_2} w_{t_3} w_{t_4} \\
&= \mathcal{B}_\mathcal{A} \w^4.
\end{aligned}
$$
Therefore, for all $\w \in \mathbb{R}^{m+n}$,
\begin{equation}\label{2.4}
\mathcal{B}_\mathcal{A} \w^3 = \bar{\mathcal{B}_\mathcal{A}} \w^3,\quad \mathcal{B}_\mathcal{A} \w^4 = \bar{\mathcal{B}_\mathcal{A}} \w^4.
\end{equation}

We now formulate a generalized tensor eigenvalue problem associated with
the lifted symmetric tensor $\bar{\mathcal B}_{\mathcal A}$. In Definition~\ref{df2.1}, this corresponds to $p=4,q=2$ and $N=m+n$.
To allow different weights for the two blocks corresponding to $\mathbf x$ and $\mathbf y$, for
$\alpha,\beta>0$, define
\begin{equation}\label{11}
\mathbf B_{\alpha,\beta}:=
\operatorname{diag}
\bigl(2\beta^2\mathbf I_m,\,2\alpha^2\mathbf I_n\bigr).
\end{equation}
Then $\mathbf B_{\alpha,\beta}$ is a symmetric positive definite
matrix in $\mathbb R^{(m+n)\times(m+n)}$.

\bd\label{df2.2}
Let  $\mathcal T\in\mathbb R^{[4,m+n]}$, a pair $(\lambda,\mathbf w)\in
\mathbb R\times(\mathbb R^{m+n}\setminus\{\mathbf 0\})$
is called a mode-$k$ $\mathbf B_{\alpha,\beta}$-eigenpair of $\mathcal T$ if
$$
\mathcal T^{(k)}\mathbf w^3
=
\lambda\mathbf B_{\alpha,\beta}\mathbf w,
\qquad
\mathbf w^\top\mathbf B_{\alpha,\beta}\mathbf w=1.
$$
When $\mathcal T$ is symmetric, $\mathcal T^{(k)}\mathbf w^3$ is independent
of $k$. In this case, we simply call $(\lambda,\mathbf w)$ a $\mathbf B_{\alpha,\beta}$-eigenpair of $\mathcal T$.
\ed

Since $\bar{\mathcal B}_{\mathcal A}$ is symmetric, $(\lambda,\mathbf w)$ is a $\mathbf B_{\alpha,\beta}$-eigenpair of $\bar{\mathcal B}_{\mathcal A}$ if
$$
\bar{\mathcal B}_{\mathcal A}\mathbf w^3=\lambda\mathbf B_{\alpha,\beta}\mathbf w,
\qquad
\mathbf w^\top\mathbf B_{\alpha,\beta}\mathbf w=1.
$$
When $\alpha=\beta=\frac1{\sqrt2}$, one has $\mathbf B_{\alpha,\beta}=\mathbf I_{m+n}$, and the above system reduces to
$$
\bar{\mathcal B}_{\mathcal A}\mathbf w^3=\lambda\mathbf w,
\qquad
\mathbf w^\top\mathbf w=1,
$$
which is precisely the standard $Z$-eigenvalue problem. Hence, the standard $Z$-eigenvalue formulation is a special case of the present weighted formulation.


\begin{proposition}\label{pro1}
Let $\bar{\mathcal{B}_\mathcal{A}} \in \mathbb{R}^{[4,m+n]}$ be symmetric, and let
${\bf B}_{\alpha,\beta}=\operatorname{diag}(2\beta^2 {\bf I}_m,\,2\alpha^2 {\bf I}_n)$, $\alpha,\beta> 0$. Then, $\bar{\mathcal{B}_\mathcal{A}}$ has at least one ${\bf B}_{\alpha,\beta}$-eigenpair. Moreover, if $(\lambda,\w)$ is a ${\bf B}_{\alpha,\beta}$-eigenpair of $\bar{\mathcal{B}_\mathcal{A}}$, then $\lambda=\bar{\mathcal{B}_\mathcal{A}}\w^4$.
\end{proposition}

\proof Consider the following optimization problem
\begin{equation}\label{1}
\max\{\bar{\mathcal{B}_\mathcal{A}}\w^4:\w^\top {\bf B}_{\alpha,\beta} \w = 1\}.
\end{equation}
Since $\alpha,\beta>0$, the matrix ${\bf B}_{\alpha,\beta}$ is symmetric positive definite. Therefore, the feasible set
\[
\mathbb{S}:=\{\mathbf w\in\mathbb R^{m+n}:\w^\top {\bf B}_{\alpha,\beta}\w=1\}
\]
is nonempty and compact.
Moreover, the objective function $\bar{\mathcal{B}_\mathcal{A}}\w^4$ is continuous on $\w\in\mathbb R^{m+n}$. By the Weierstrass theorem, problem \eqref{1} admits an optimal solution $\mathbf w^*\in \mathbb{S}$.

The Lagrangian function of \eqref{1} is given by
\begin{equation}
L(\w,\mu)=\bar{\mathcal{B}_\mathcal{A}}\w^4-\mu(\w^\top {\bf B}_{\alpha,\beta}\w-1).
\end{equation}
By the Lagrange multiplier theorem, for the optimal solution $\w^*$, there exists $\mu^*\in\mathbb R$ such that
\[
\nabla_{\mathbf w}L(\mathbf w^*,\mu^*)={\bf 0},
\qquad
\nabla_{\mu L}(\mathbf w^*,\mu^*)=0.
\]
Since $\bar{\mathcal{B}}_\mathcal{A}$ is symmetric, it follows that
$$
\begin{cases}
\nabla_{\w}L(\w^*,\mu^*)=4\bar{\mathcal{B}_\mathcal{A}}{\w^*}^3-2\mu^*{\bf B}_{\alpha,\beta}{\w^*}={\bf 0},\\[10pt]
\nabla_{\mu}L(\w^*,\mu^*)=({\w^*})^\top {\bf B}_{\alpha,\beta}{\w^*}-1=0.
\end{cases}
$$
Let $\lambda^*:=\mu^*/2$. Then
\[
\bar{\mathcal{B}}_A\mathbf w^{*3}=\lambda^* {\bf B}_{\alpha,\beta}\mathbf w^*,
\qquad
(\mathbf w^*)^\top {\bf B}_{\alpha,\beta}\mathbf w^*=1,
\]
which shows that $(\lambda^*,\mathbf w^*)$ is a $\mathbf B_{\alpha,\beta}$-eigenpair of
$\bar{\mathcal{B}}_{\mathcal A}$.

Now let $(\lambda,\w)$ be any ${\bf B}_{\alpha,\beta}$-eigenpair of $\bar{\mathcal{B}}_\mathcal{A}$. Then,
$\bar{\mathcal{B}}_\mathcal{A}\w^3=\lambda {\bf B}_{\alpha,\beta}\w$.
Multiplying both sides by $\w^\top$ gives
\[
\bar{\mathcal{B}}_\mathcal{A}\w^4=\lambda\,\w^\top {\bf B}_{\alpha,\beta}\w.
\]
Since $\w^\top {\bf B}_{\alpha,\beta}\w=1$, we conclude that
$$
\lambda=\bar{\mathcal{B}}_\mathcal{A}\w^4.
$$ \qed

By \eqref{2.4}, we obtain the following result.


\begin{proposition}\label{prop2.1}
Let \(\mathcal A=(a_{ijkl})\in\mathbb P^{m\times n\times m\times n}\), and let
\(\mathcal B_{\mathcal A}\) and \(\bar{\mathcal B}_{\mathcal A}\) be defined by
\eqref{2.1} and \eqref{2.3}, respectively. Then, for any
\((\lambda,\mathbf w)\in\mathbb R\times(\mathbb R^{m+n}\setminus\{\mathbf0\})\),
\[
\mathcal B_{\mathcal A}^{(1)}\mathbf w^3
=
\lambda\mathbf B_{\alpha,\beta}\mathbf w,
\qquad
\mathbf w^\top\mathbf B_{\alpha,\beta}\mathbf w=1
\]
if and only if
\[
\bar{\mathcal B}_{\mathcal A}\mathbf w^3
=
\lambda\mathbf B_{\alpha,\beta}\mathbf w,
\qquad
\mathbf w^\top\mathbf B_{\alpha,\beta}\mathbf w=1.
\]
\end{proposition}

In view of Proposition~\ref{prop2.1}, it suffices to establish the relationship between the M-eigenpairs of $\mathcal{A}$ and the mode-$1$ $\mathbf{B}_{\alpha,\beta}$-eigenpairs of ${\mathcal{B}_{\mathcal{A}}}$.

\bt\label{th2.3}
Given $\mathcal{A}=(a_{ijkl}) \in \mathbb{P}^{m\times n \times m\times n}$, let
${\mathcal B}_{\mathcal A}\in\mathbb R^{[4,m+n]}$ be defined by \eqref{2.1}, and define ${\bf B}_{\alpha,\beta}=\operatorname{diag}(2\beta^2 {\bf I}_m,\,2\alpha^2 {\bf I}_n)$, where $\alpha,\beta>0$. Then

(1) If $(\lambda, \x, \y)$ is an $M$-eigenpair of $\mathcal{A}$, then $\left(\frac{\lambda}{8\alpha^2\beta^2}, (\frac{1}{2\beta}\x^\top,\frac{1}{2\alpha}\y^\top)^\top\right)$ is a mode-$1$  ${\bf B}_{\alpha,\beta}$-eigenpair of $\mathcal{B}_{\mathcal{A}}$.

(2) If $(\lambda, \w)$ is a mode-$1$ ${\bf B}_{\alpha,\beta}$-eigenpair of $\mathcal{B}_{\mathcal{A}}$ with $\lambda \neq 0$, $\w = (\w_x^\top, \w_y^\top)^\top$, $\w_x = (w_1, \dots, w_m)^\top$ and $\w_y = (w_{m+1}, \dots, w_{m+n})^\top$, then $\left(8\alpha^2\beta^2\lambda, 2\beta\w_x, 2\alpha\w_y\right)$ is an $M$-eigenpair of $\mathcal{A}$.

(3)  Let $(0, \w)$ be a mode-$1$ ${\bf B}_{\alpha,\beta}$-eigenpair of $\mathcal{B}_{\mathcal{A}}$, where $\w = (\w_x^\top, \w_y^\top)^\top$, $\w_x = (w_1, \dots, w_m)^\top$ and $\w_y = (w_{m+1}, \dots, w_{m+n})^\top$. If $\w_x\neq\bf 0$ and $\w_y\neq\bf 0$, then $\left(0, \frac{\w_x}{\|\w_x\|_2}, \frac{\w_y}{\|\w_y\|_2}\right)$ is an $M$-eigenpair of $\mathcal{A}$. If either $\w_x=\bf 0$ or $\w_y=\bf 0$, this zero eigenpair of the lifted tensor does not correspond to an $M$-eigenpair of $\mathcal{A}$.
\et

\proof (1) Let $(\lambda, \x, \y)$ be an $M$-eigenpair of $\mathcal{A}$. Then
$$
\mathcal{A}\cdot\y\x\y=\lambda \x,\qquad \mathcal{A}\x\y\x\cdot=\lambda \y,\qquad \x^\top \x=\y^\top \y=1.
$$
Define
$\w=(\frac{1}{2\beta}\x^\top,\frac{1}{2\alpha}\y^\top)^\top \in \mathbb{R}^{m+n}$. By \eqref{2.2} and \eqref{2.4}, we have
$$
\begin{aligned}
\mathcal{B}_{\mathcal{A}} \w^3 &=
\begin{pmatrix}
{\mathcal{A}(\frac{\y}{2\alpha})(\frac{\x}{2\beta})(\frac{\y}{2\alpha})} \\
{\mathcal{A}(\frac{\x}{2\beta})(\frac{\y}{2\alpha})(\frac{\x}{2\beta})}
\end{pmatrix}
=\begin{pmatrix}
\frac{1}{8\alpha^2\beta}\mathcal{A}\cdot\y\x\y  \\
\frac{1}{8\alpha\beta^2}\mathcal{A}\x\y\x\cdot
\end{pmatrix}
=\frac{1}{8\alpha^2\beta^2}
\begin{pmatrix}
\beta\mathcal{A}\cdot\y\x\y \\
\alpha\mathcal{A}\x\y\x\cdot
\end{pmatrix}\\
&=\frac{\lambda}{8\alpha^2\beta^2}
\begin{pmatrix}
2\beta^2 \boldsymbol{I}_m & \boldsymbol{0} \\
\boldsymbol{0} & 2\alpha^2 \boldsymbol{I}_n
\end{pmatrix}
\begin{pmatrix}
\frac{1}{2\beta} \x \\
\frac{1}{2\alpha} \y
\end{pmatrix}
=\frac{\lambda}{8\alpha^2\beta^2}{\bf B}_{\alpha,\beta}\w.
\end{aligned}
$$
and
$$
\w^\top {\bf B}_{\alpha,\beta}\w=\left( \frac{1}{2\beta} \x, \frac{1}{2\alpha}\y\right)
\begin{pmatrix}
2\beta^2 \boldsymbol{I}_m & \boldsymbol{0} \\
\boldsymbol{0} & 2\alpha^2 \boldsymbol{I}_n
\end{pmatrix}
\begin{pmatrix}
\frac{1}{2\beta} \x \\
\frac{1}{2\alpha} \y
\end{pmatrix}
=\frac{1}{2} \x^\top \x + \frac{1}{2}\y^\top\y = 1,
$$
it can be seen that $\left(\frac{\lambda}{8\alpha^2\beta^2},\w\right)$ is a mode-$1$ ${\bf B}_{\alpha,\beta}$-eigenpair of $\mathcal{B}_{\mathcal{A}}$.

(2) Let $(\lambda,\w)$ be a mode-$1$ ${\bf B}_{\alpha,\beta}$-eigenpair of $\mathcal B_{\mathcal A}$ with $\lambda\neq 0$, where $\w=(\w_x^\top,\w_y^\top)^\top$, and $\w_x=(w_1,\dots,w_m)^\top,\w_y=(w_{m+1},\dots,w_{m+n})^\top$. By Proposition~\ref{prop2.1} and \eqref{2.2}, it follows that
$$
\bar{\mathcal B}_{\mathcal A}\mathbf w^3=\lambda\mathbf B_{\alpha,\beta}\mathbf w
$$
is equivalent to
$$
\begin{pmatrix}
\mathcal A\cdot\mathbf w_y\mathbf w_x\mathbf w_y\\
\mathcal A\mathbf w_x\mathbf w_y\mathbf w_x\cdot
\end{pmatrix}
=\lambda\begin{pmatrix}
2\beta^2\mathbf w_x\\
2\alpha^2\mathbf w_y
\end{pmatrix}.
$$
That is,
\begin{equation}\label{eq:lifted-block}
\mathcal A\cdot\mathbf w_y\mathbf w_x\mathbf w_y
=2\lambda\beta^2\mathbf w_x,\qquad
\mathcal A\mathbf w_x\mathbf w_y\mathbf w_x\cdot
=2\lambda\alpha^2\mathbf w_y .
\end{equation}

We first show that both $\w_x$ and $\w_y$ are nonzero vectors. If $\w_y={\bf 0}$, then the
first equation in \eqref{eq:lifted-block} reduces to
$2\lambda\beta^2\mathbf w_x=\mathbf0$. Since $\lambda\neq0$ and $\beta\neq0$, it follows that $\w_x=\mathbf0$, contradicting $\w\neq\mathbf0$. Hence $\w_y\neq\mathbf0$. Similarly,
$\w_x\neq\mathbf0$.

Multiplying the first equation in \eqref{eq:lifted-block} by
\(\mathbf w_x^\top\) and the second one by $\w_y^\top$, we obtain
$$
\mathcal A\mathbf w_x\mathbf w_y\mathbf w_x\mathbf w_y
=2\lambda\beta^2\|\mathbf w_x\|^2,
\qquad
\mathcal A\mathbf w_x\mathbf w_y\mathbf w_x\mathbf w_y
=2\lambda\alpha^2\|\mathbf w_y\|^2.
$$
Since $\lambda\neq0$, this implies
$
\beta^2\|\mathbf w_x\|^2=\alpha^2\|\mathbf w_y\|^2
$.
Together with the normalization
$$
\mathbf w^\top\mathbf B_{\alpha,\beta}\mathbf w
=2\beta^2\|\mathbf w_x\|^2+2\alpha^2\|\mathbf w_y\|^2=1,
$$
therefore implies
$$
2\beta^2\|\mathbf w_x\|^2=2\alpha^2\|\mathbf w_y\|^2=\frac12.
$$

Let $\x = 2\beta\w_x$ and $\y =2\alpha\w_y$. Then $\w_x =\frac{\x}{2\beta}$, $\w_y =\frac{\y}{2\alpha}$ and $\|\x\|_2 = \|\y\|_2 = 1$. From
\[
\begin{aligned}
\begin{pmatrix}
\mathcal{A} \y\x \y \\
\mathcal{A}\x \y\x
\end{pmatrix}
&=
\begin{pmatrix}
\mathcal{A}\cdot(2\alpha\w_y)(2\beta\w_x)(2\alpha\w_y) \\
\mathcal{A}(2\beta\w_x)(2\alpha\w_y)(2\beta\w_x)\cdot
\end{pmatrix}
=
\begin{pmatrix}
8\alpha^2\beta\mathcal{A}\cdot\w_y\w_x\w_y \\
8\alpha\beta^2\mathcal{A}\w_x\w_y\w_x\cdot
\end{pmatrix}\\
&=
\begin{pmatrix}
8\alpha^2\beta 2\lambda\beta^2\w_x \\
8\alpha\beta^2 2\lambda\alpha^2\w_y
\end{pmatrix}= 8\alpha^2\beta^2\lambda
\begin{pmatrix}
2\beta\w_x \\
2\alpha\w_y
\end{pmatrix}
= 8\alpha^2\beta^2\lambda
\begin{pmatrix}
\x \\
\y
\end{pmatrix},
\end{aligned}
\]
it can be seen that $(8\alpha^2\beta^2\lambda, 2\beta\w_x,2\alpha\w_y)$ is an $M$-eigenpair of $\mathcal{A}$.

(3) Let $(0,\mathbf w)$ be a mode-$1$ ${\bf B}_{\alpha,\beta}$-eigenpair of $\mathcal{B}_{\mathcal{A}}$, where $\w = (\w_x^\top, \w_y^\top)^\top$, $\w_x = (w_1, \dots, w_m)^\top$ and $\w_y = (w_{m+1}, \dots, w_{m+n})^\top$. Then, by \eqref{2.2}, we have
\begin{equation}\label{2.7}
\lambda{\bf B}_{\alpha,\beta}
\begin{pmatrix}
\w_x \\
\w_y
\end{pmatrix}
= \lambda{\bf B}_{\alpha,\beta} \w = \mathcal{B}_{\mathcal{A}} \w^3=
\begin{pmatrix}
\mathcal{A} \cdot\w_y \w_x\w_y \\
\mathcal{A} \w_x\w_y\w_x\cdot
\end{pmatrix}.
\end{equation}

If both $\w_x$ and $\w_y$ are nonzero vectors. Let $\x = \frac{\w_x}{\|\w_x\|_2}$ and $\y = \frac{\w_y}{\|\w_y\|_2}$. Then $\|\x\|_2 = \|\y\|_2 = 1$. From
\[
\begin{pmatrix}
\mathcal{A}\cdot\y\x \y \\
\mathcal{A}\x \y\x\cdot
\end{pmatrix}
=\begin{pmatrix}
\frac{\mathcal{A} \w_y \w_x \w_y}{\|\w_y\|_2\|\w_x\|_2\|\w_y\|_2} \\
\frac{\mathcal{A} \w_x \w_y\w_x} {\|\w_x\|_2 \|\w_y\|_2\|\w_x\|_2}
\end{pmatrix}
=\lambda{\bf B}_{\alpha,\beta}
\begin{pmatrix}
\frac{ \w_x}{\|\w_x\|_2\|\w_y\|_2^2} \\
\frac{\w_y}{\|\w_x\|_2^2 \|\w_y\|_2}
\end{pmatrix}
=
\begin{pmatrix}
0 \\
0
\end{pmatrix}
= 0
\begin{pmatrix}
\x \\
\y
\end{pmatrix},
\]
it can be seen that $\lambda = 0$ is an $M$-eigenvalue of $\mathcal{A}$ with associated $M$-eigenvectors $(\x, \y)$.

Now suppose that either $\w_x$ or $\w_y$ is a zero vector. If $\w_x=\bf 0$, by $\w \in \mathbb{R}^{m+n} \setminus \{\bf 0\}$, then $\w_y\neq\bf 0$. Moreover, it follows from \eqref{2.7} that if $(\x,\y)$ is a pair of $M$-eigenvectors of $\mathcal{A}$, then $\x = 2\beta\w_x$, where $\beta\neq0$. By $\|\w_x\|_2 = 0$, we have $\|\x\|_2 = 0$, which implies that the zero eigenpair $(0,\w)$ of the lifted tensor cannot generate an $M$-eigenpair through the above recovery procedure. Similarly, the case $\w_y=\mathbf0$ can be treated analogously.
\qed

Combining Proposition~\ref{prop2.1} with Theorem~\ref{th2.3}, we obtain the corresponding result.

\begin{corollary}\label{th2.3}
Let $\mathcal{A}=(a_{ijkl}) \in \mathbb{P}^{m\times n \times m\times n}$, and let
\(\bar{\mathcal B}_{\mathcal A}\in\mathbb R^{[4,m+n]}\) be defined by \eqref{2.3}.

(1) If $(\lambda, \x, \y)$ is an $M$-eigenpair of $\mathcal{A}$, then $\left(\frac{\lambda}{8\alpha^2\beta^2}, (\frac{1}{2\beta}\x^\top,\frac{1}{2\alpha}\y^\top)^\top\right)$  is a ${\bf B}_{\alpha,\beta}$-eigenpair of $\bar{\mathcal{B}_{\mathcal{A}}}$.

(2) If $(\lambda, \w)$ is a ${\bf B}_{\alpha,\beta}$-eigenpair of $\bar{\mathcal{B}_{\mathcal{A}}}$ with $\lambda \neq 0$, where $\w = (\w_x^\top,\w_y^\top)^\top$, and $\w_x = (w_1, \dots, w_m)^\top$, $\w_y = (w_{m+1}, \dots, w_{m+n})^\top$, then $\left(8\alpha^2\beta^2\lambda, 2\beta\w_x, 2\alpha\w_y\right)$ is an $M$-eigenpair of $\mathcal{A}$.

(3)  Let $(0, \w)$ be a ${\bf B}_{\alpha,\beta}$-eigenpair of $\bar{\mathcal{B}_{\mathcal{A}}}$, where $\w = (\w_x^\top,\w_y^\top)^\top$, $\w_x = (w_1, \dots, w_m)^\top$ and $\w_y = (w_{m+1}, \dots, w_{m+n})^\top$. If $\w_x\neq\bf 0$ and $\w_y\neq\bf 0$, then $\left(0, \frac{\w_x}{\|\w_x\|_2}, \frac{\w_y}{\|\w_y\|_2}\right)$ is an $M$-eigenpair of $\mathcal{A}$. If either $\w_x=\bf 0$ or $\w_y=\bf 0$, then this zero eigenpair of the lifted tensor does not correspond to an $M$-eigenpair of $\mathcal{A}$.
\end{corollary}

Based on the equivalence established in Corollary \ref{th2.3}, it is evident that the $M$-eigenpairs of $\mathcal A$ can be derived by computing the ${\bf B}_{\alpha,\beta}$-eigenpairs of the lifted symmetric tensor $\bar{\mathcal B}_{\mathcal A}$. The procedure is summarized in Algorithm \ref{alg1}.

\begin{algorithm}[h]
\caption{The lifted \(\mathbf B_{\alpha,\beta}\)-eigenpair method for computing \(M\)-eigenpairs}
\label{alg1}
\begin{algorithmic}[1]
\Require A tensor \(\mathcal A\in\mathbb P^{m\times n\times m\times n}\), parameters
\(\alpha,\beta>0\).
\Ensure A set \(\mathcal M\) of real \(M\)-eigenpairs of \(\mathcal A\).

\State Construct the lifted tensor \(\mathcal B_{\mathcal A}\in\mathbb R^{[4,m+n]}\) by \eqref{2.1}.
\State Construct the symmetric tensor \(\bar{\mathcal B}_{\mathcal A}\in\mathbb R^{[4,m+n]}\) by \eqref{2.3}.
\State Set
\[
\mathbf B_{\alpha,\beta}=\operatorname{diag}(2\beta^2\mathbf I_m,\,2\alpha^2\mathbf I_n).
\]
\State Compute the set \(\mathcal E\) of real \(\mathbf B_{\alpha,\beta}\)-eigenpairs
\((\mu,\mathbf w)\) of \(\bar{\mathcal B}_{\mathcal A}\).
\State Initialize \(\mathcal M\leftarrow\emptyset\).

\For{each \((\mu,\mathbf w)\in\mathcal E\)}
    \State Partition \(\mathbf w=(\mathbf w_x^\top,\mathbf w_y^\top)^\top\), where
    \(\mathbf w_x\in\mathbb R^m\) and \(\mathbf w_y\in\mathbb R^n\).
    \If{\(\mu\neq0\)}
        \State Set
        \[
        \lambda\leftarrow 8\alpha^2\beta^2\mu,\qquad
        \mathbf x\leftarrow 2\beta\mathbf w_x,\qquad
        \mathbf y\leftarrow 2\alpha\mathbf w_y.
        \]
        \State Add \((\lambda,\mathbf x,\mathbf y)\) to \(\mathcal M\).
    \Else
        \If{\(\mathbf w_x\neq\mathbf0\) and \(\mathbf w_y\neq\mathbf0\)}
            \State Set
            \[
            \lambda\leftarrow0,\qquad
            \mathbf x\leftarrow\frac{\mathbf w_x}{\|\mathbf w_x\|},\qquad
            \mathbf y\leftarrow\frac{\mathbf w_y}{\|\mathbf w_y\|}.
            \]
            \State Add \((\lambda,\mathbf x,\mathbf y)\) to \(\mathcal M\).
        \Else
            \State Discard this solution.
        \EndIf
    \EndIf
\EndFor

\State \Return \(\mathcal M\).
\end{algorithmic}
\end{algorithm}

\section{Perturbation bounds for the largest $M$-eigenvalue}
In this section, we derive a perturbation bound for the largest $M$-eigenvalue under additive perturbations and establish its exact counterpart in the lifted $\mathbf B_{\alpha,\beta}$-eigenvalue formulation.

Let $\mathcal{A}\in\mathbb P^{m\times n\times m\times n}$ and let $\mathcal E\in\mathbb P^{m\times n\times m\times n}$ be a perturbation tensor. Set $\tilde{\mathcal{A}}=\mathcal{A}+\mathcal{E}$. Based on (1.2), we first derive the following perturbation bound.

\begin{theorem}\label{th3.1}
Let $\mathcal A,\mathcal E\in\mathbb P^{m\times n\times m\times n}$, and set perturbed tensor $\widetilde{\mathcal A}=\mathcal A+\mathcal E$. Then
\begin{equation}\label{3.1}
\lambda_{\max}^{M}(\mathcal{A}) +\lambda_{\min}^{M}(\mathcal{E}) \leq \lambda_{\max}^{M}(\tilde{\mathcal{A}}) \leq \lambda_{\max}^{M}(\mathcal{A}) + \lambda_{\max}^{M}(\mathcal{E}).
\end{equation}
\end{theorem}

\proof By the variational characterization \eqref{1.2}, we have
\begin{equation}\label{3.2}
\begin{aligned}
\lambda_{\max}^{M}(\tilde{\mathcal{A}}) &= \max\{\tilde{\mathcal{A}}\x\y\x\y: \x^\top\x = 1,\ \y^\top\y= 1\} \\
&= \max\{\mathcal{A}\x\y\x\y+\mathcal{E}\x\y\x\y:\x^\top\x = 1,\ \y^\top\y= 1,\} \\
&\leq \max\{\mathcal{A}\x\y\x\y:\x^\top\x = 1,\ \y^\top\y= 1\} + \max\{\mathcal{E}\x\y\x\y:\x^\top\x = 1,\ \y^\top\y= 1\} \\
&= \lambda_{\max}^{M}(\mathcal{A}) + \lambda_{\max}^{M}(\mathcal{E}),
\end{aligned}
\end{equation}
which proves the upper bound.

On the other hand, let $(\x,\y)$ be a pair of unit $M$-eigenvectors associated with $\lambda_{M\max}(\mathcal{A})$. Then
\begin{equation}\label{3.3}
\begin{aligned}
\lambda_{\max}^{M}(\tilde{\mathcal{A}}) &\geq \tilde{\mathcal{A}}\x\y\x\y= \mathcal{A}\x\y\x\y+ \mathcal{E}\x\y\x\y= \lambda_{\max}^{M}(\mathcal{A}) +\mathcal{E}\x\y\x\y\geq \lambda_{\max}^{M}(\mathcal{A})+\lambda_{\min}^{M}(\mathcal{E}).
\end{aligned}
\end{equation}

Combining \eqref{3.2} and \eqref{3.3}, we get the perturbation bound \eqref{3.1}.
\qed

Using the eigenvalue correspondence established in Theorem~\ref{th2.3}, the perturbation bound in
Theorem~\ref{th3.1} can be equivalently expressed in the lifted formulation. In what follows,
$\lambda_{\min}^{B}$ and $\lambda_{\max}^{B}$ denote the smallest and largest recoverable real
$\mathbf B_{\alpha,\beta}$-eigenvalues, respectively, of the corresponding symmetric lifted tensor.



\begin{corollary}\label{co3.1}
Let \(\mathcal A,\mathcal E\in\mathbb P^{m\times n\times m\times n}\),
and set perturbed tensor \(\widetilde{\mathcal A}=\mathcal A+\mathcal E\). Then
\[
\lambda_{\max}^{M}(\mathcal A)
+
8\alpha^2\beta^2\lambda_{\min}^{B}(\mathcal B_{\mathcal E})
\le
\lambda_{\max}^{M}(\widetilde{\mathcal A})
\le
\lambda_{\max}^{M}(\mathcal A)
+
8\alpha^2\beta^2\lambda_{\max}^{B}(\mathcal B_{\mathcal E}).
\]
\end{corollary}

\proof
Since the lifting map is linear, we have
\[
\mathcal B_{\widetilde{\mathcal A}}
=
\mathcal B_{\mathcal A}
+
\mathcal B_{\mathcal E}.
\]
Moreover, by the eigenvalue correspondence in Theorem \ref{th2.3}, applied to
the perturbation tensor \(\mathcal E\),
\[
\lambda_{\max}^{M}(\mathcal E)
=
8\alpha^2\beta^2\lambda_{\max}^{B}(\mathcal B_{\mathcal E}),
\qquad
\lambda_{\min}^{M}(\mathcal E)
=
8\alpha^2\beta^2\lambda_{\min}^{B}(\mathcal B_{\mathcal E}).
\]
Substituting these identities into Theorem 3.1 gives the desired
bound.
\(\square\)

\begin{Remark}
Corollary \ref{co3.1} is not a sharper perturbation estimate than Theorem \ref{th3.1}. Rather,  it
gives its exact representation in the lifted formulation. Indeed, the linearity of the lifting map ensures that additive perturbations of $\mathcal A$ remain additive after lifting, while the eigenvalue correspondence shows that the perturbation interval is preserved up to the scaling factor $8\alpha^2\beta^2$. Therefore, the proposed transformation introduces no additional relaxation with respect to the derived perturbation bounds.
\end{Remark}

For comparison with the standard $Z$-eigenvalue normalization, we introduce two spectral radii for the lifting construction $\mathcal B_{\mathcal A}$. Define
$$
\rho_Z(\mathcal B_{\mathcal A}):=\max\left\{|\mathcal B_{\mathcal A}\mathbf z^4|:\ \mathbf z^\top\mathbf z=1\right\},
$$
and
$$
\rho_B(\mathcal B_{\mathcal A}):=\max\left\{|\mathcal B_{\mathcal A}\mathbf w^4|:\ \mathbf w^\top\mathbf B_{\alpha,\beta}\mathbf w=1\right\},
$$
where $\mathbf B_{\alpha,\beta}$ is defined in
\eqref{11}. Here, $\rho_Z({\mathcal B}_{\mathcal A})$ is the standard $Z$-spectral radius,
whereas $\rho_{B}({\mathcal B}_{\mathcal A})$ is its weighted counterpart under the $\mathbf B_{\alpha,\beta}$-normalization.


The following lemma establishes their exact scaling relation.

\begin{lemma}\label{lem3.1}
Let $\mathcal A\in\mathbb P^{m\times n\times m\times n}$, and let
${\mathcal B}_{\mathcal A}\in \mathbb{R}^{[4, m+n]}$ be the associated lifted tensor. Then
$$\rho_B(\mathcal{B}_{\mathcal A})=\frac{1}{4\alpha^2\beta^2} \rho_Z(\mathcal{B}_{\mathcal A}).$$
\end{lemma}

\proof Let $\mathbf z=(\mathbf p^\top,\mathbf q^\top)^\top$ with $\|\mathbf z\|=1$,
write $\mathbf p=r\boldsymbol\xi$ and $\mathbf q=s\boldsymbol\eta$, where
$r,s\ge 0$ and $\|\boldsymbol\xi\|=\|\boldsymbol\eta\|=1$ whenever
$\mathbf p, \mathbf q \neq \bf 0$. If $\mathbf p=\bf 0$ or $\mathbf q=\bf 0$, then clearly $\mathcal B_{\mathcal A}\mathbf z^4=0$, and the claimed identity holds immediately. Thus, it is
sufficient to consider the case $\mathbf p\neq \bf 0$ and $\mathbf q\neq \bf 0$.
Since $\mathcal B_{\mathcal A}$ is biquadratic with respect to the two blocks, we have
$$
\mathcal B_{\mathcal A}\mathbf z^4=r^2s^2\mathcal B_{\mathcal A}\boldsymbol\xi\boldsymbol\eta\boldsymbol\xi\boldsymbol\eta .
$$
Together with $\|\mathbf z\|=1$, we have $r^2+s^2=1$. Thus
$$
\max_{r^2+s^2=1} r^2s^2=\frac14.
$$
It follows that
$$
\rho_Z(\mathcal B_{\mathcal A})
=\frac14
\max_{\|\boldsymbol\xi\|=\|\boldsymbol\eta\|=1}
|\mathcal B_{\mathcal A}\boldsymbol\xi\boldsymbol\eta\boldsymbol\xi\boldsymbol\eta|.
$$

Similarly, let $\mathbf w=(\mathbf u^\top,\mathbf v^\top)^\top$ satisfying
$\mathbf w^\top\mathbf B_{\alpha,\beta}\mathbf w=1$, write
$\mathbf u=a\boldsymbol\xi$ and $\mathbf v=b\boldsymbol\eta$, where
$a,b\geq 0$ and $\|\boldsymbol\xi\|=\|\boldsymbol\eta\|=1$ whenever
$\mathbf u, \mathbf v\neq \bf 0$.  If $\mathbf u=0$ or $\mathbf v=0$, then
$\mathcal B_{\mathcal A}\mathbf w^4=0$. Thus, we only need to consider the case
$\mathbf u\neq0$ and $\mathbf v\neq0$. By the definition of
$\mathbf B_{\alpha,\beta}$, we have
$$
2\beta^2a^2+2\alpha^2b^2=1,
$$
Moreover, by the biquadratic block structure of $\mathcal B_{\mathcal A}$, we obtain
\[
\mathcal B_{\mathcal A}\mathbf w^4
=
a^2b^2
\mathcal B_{\mathcal A}\boldsymbol\xi\boldsymbol\eta\boldsymbol\xi\boldsymbol\eta.
\]
Since
\[
\max_{2\beta^2a^2+2\alpha^2b^2=1}a^2b^2
=
\frac{1}{16\alpha^2\beta^2},
\]
where the maximum is attained at
$
2\beta^2a^2=2\alpha^2b^2=\frac12
$.
Therefore,
$$
\rho_{B}(\mathcal B_{\mathcal A})=\frac{1}{16\alpha^2\beta^2}
\max_{\|\boldsymbol\xi\|=\|\boldsymbol\eta\|=1}
|\mathcal B_{\mathcal A}\boldsymbol\xi\boldsymbol\eta\boldsymbol\xi\boldsymbol\eta|.
$$
Combining the above two identities gives
$$
\rho_{B}(\mathcal B_{\mathcal A})
=
\frac{1}{4\alpha^2\beta^2}\rho_Z(\mathcal B_{\mathcal A}).
$$

\qed

\begin{corollary}\label{co3.2}
Let $\mathcal A,\mathcal E\in \mathbb P^{m\times n\times m\times n}$, and set
$\widetilde{\mathcal A}:=\mathcal A+\mathcal E$.
Then
$$|\lambda_{M\max}(\tilde{\mathcal{A}}) - \lambda_{M\max}(\mathcal{A})| \le 8\alpha^2\beta^2 \rho_B(\mathcal{B}_{\mathcal{E}})= 2\rho_Z(\mathcal{B}_{\mathcal{E}}).$$
\end{corollary}

\proof
It follows from Corollary \ref{co3.1} that
$$
\begin{aligned}
\bigl|
\lambda_{M\max}(\tilde{\mathcal A})-\lambda_{M\max}(\mathcal A)
\bigr|
&\le
8\alpha^2\beta^2
\max\left\{
|\lambda_{B\min}(\mathcal B_{\mathcal E})|,
|\lambda_{B\max}(\mathcal B_{\mathcal E})|
\right\}\\
&=8\alpha^2\beta^2\rho_B(\mathcal B_{\mathcal E}).
\end{aligned}
$$
Applying Lemma~\ref{lem3.1} to the perturbation tensor \(\mathcal E\), we obtain
$$
8\alpha^2\beta^2\rho_B(\mathcal B_{\mathcal E})
=2\rho_Z(\mathcal B_{\mathcal E}),
$$
which completes the proof.\qed

\begin{Remark}
When $\alpha=\beta=\frac{1}{\sqrt2}$, one has $\mathbf B_{\alpha,\beta}=\mathbf I_{m+n}$. Consequently, the proposed $\mathbf B_{\alpha,\beta}$-eigenvalue formulation reduces to the standard $Z$-eigenvalue formulation. In this sense, the $Z$-eigenvalue lifting is a special case of the present framework. Moreover, Corollary \ref{3.2} shows that the parameterized lifting framework preserves the absolute perturbation bound exactly after the prescribed scaling and introduces no additional relaxation.
\end{Remark}

\section{Numerical experiments}

In this section, we construct several numerical examples to evaluate the performance of the lifted $\mathbf B_{\alpha,\beta}$-eigenpair method in Algorithm~\ref{alg1}. The first two examples illustrate the recovery of the complete real $M$-spectrum, the next two examine the effects
of the normalization parameters and tensor dimensions, and the final example illustrates the perturbation bounds established in Section~3.


All numerical experiments were implemented in MATLAB 9.0 on a personal computer with AMD Ryzen 7 4800H CPU 2.90GHz and 16 GB random-access memory (RAM).

To measure the accuracy of a computed $M$-eigenpair $(\lambda, \mathbf{x}, \mathbf{y})$ for a given tensor $\mathcal{A}$, we define the residual norm (Res) as follows:
\begin{equation} \label{4.1}
\text{Res} := \sqrt{\|\mathcal{A} \cdot \mathbf{yxy} - \lambda \mathbf{x}\|^2 + \|\mathcal{A} \mathbf{xyx} \cdot - \lambda \mathbf{y}\|^2}.
\end{equation}

In Examples~\ref{ex1} and~\ref{ex2}, set $\alpha=\beta=\frac{1}{\sqrt{2}}$,
so that $\mathbf B_{\alpha,\beta}=\mathbf I_{m+n}$. Consequently, the generalized eigenvalue problem for the symmetric lifted tensor reduces to the standard $Z$-eigenvalue problem. Several numerical methods are available for computing $Z$-eigenpairs of symmetric tensors; see, for example, \cite{CHZ16,CH22,LZ16,MW20}. In these two examples, we employ the method proposed in
\cite{CHZ16} to compute the real $Z$-eigenpairs of $\bar{\mathcal B}_{\mathcal A}$ and then recover the corresponding $M$-eigenpairs.

In the subsequent experiments, consider nonstandard positive choices of $\alpha$ and $\beta$ to investigate the parameterized formulation and to verify the perturbation estimates numerically.
For these cases, we directly solve the resulting generalized tensor eigenvalue problems using the method developed in \cite{CHZ16}. The computed real $\mathbf B_{\alpha,\beta}$-eigenpairs of the
corresponding symmetric lifted tensors are then converted into $M$-eigenpairs by the recovery procedure.


\begin{example}\label{ex1}
Let $\mathcal{A}= (a_{ijkl}) \in \mathbb{P}^{2\times2\times2\times 2}$, whose non-zero entries are
\[
\begin{aligned}
a_{1111} &= 1, \quad a_{1112}=a_{1211} = 2, \quad a_{1121}=a_{2111} = 2, \quad a_{1122}=a_{2112}=a_{1221}=a_{2211} = 4,\\
a_{1212} &= 3, \quad a_{1222}=a_{2212}= 5, \quad a_{2122}=a_{2221} = 5, \quad a_{2121} = 3, \quad a_{2222} = 6,
\end{aligned}
\]
and $a_{ijkl} = 0$ otherwise.
\end{example}

The results in \cite{QDH09} show that this tensor has six different $M$-eigenvalues: $-1.2764$, $0.0710$, $0.1242$, $0.2765$, $0.3437$ and $15.2091$.

To verify the effectiveness of our approach, we employ Algorithm \ref{alg1} to compute the $M$-eigenpairs of this tensor. Throughout this example, we set
$\alpha=\beta=\frac{1}{\sqrt{2}}$, so that
$\mathbf B_{\alpha,\beta}=\mathbf I_4$ and the generalized eigenvalue problem reduces to a standard $Z$-eigenvalue problem.

Let $\mathbf0_2$ denote the $2\times2$ zero matrix. For each $i,j\in[2]$, define
$\mathbf H_{2+j,i}\in\mathbb R^{2\times2}$ by
\begin{equation}\label{eq:H-example41}
\bigl(\mathbf H_{2+j,i}\bigr)_{qp}
:=
a_{pjiq},
\qquad p,q\in[2].
\end{equation}
Then the nonzero slices of
$\mathcal B_{\mathcal A}\in\mathbb R^{[4,4]}$
have the block representation
\begin{equation}\label{eq:lifted-block-ex41}
\mathcal B_{\mathcal A}(:,:,k,l)
=
\begin{cases}
\begin{pmatrix}
\mathbf0_2 & \mathbf0_2\\
\mathbf H_{kl} & \mathbf0_2
\end{pmatrix},
&
k\in\{3,4\},\quad l\in\{1,2\},\\[5mm]
\begin{pmatrix}
\mathbf0_2 & \mathbf H_{kl}\\
\mathbf0_2 & \mathbf0_2
\end{pmatrix},
&
k\in\{1,2\},\quad l\in\{3,4\},\\[5mm]
\mathbf0_{4\times4},
&\text{otherwise}.
\end{cases}
\end{equation}

The nonzero lower-left blocks in
\eqref{eq:lifted-block-ex41} are
\[
\mathbf H_{31}
=
\begin{pmatrix}
1&2\\
2&4
\end{pmatrix},
\qquad
\mathbf H_{41}
=
\begin{pmatrix}
2&4\\
3&5
\end{pmatrix},
\]
and
\[
\mathbf H_{32}
=
\begin{pmatrix}
2&3\\
4&5
\end{pmatrix},
\qquad
\mathbf H_{42}
=
\begin{pmatrix}
4&5\\
5&6
\end{pmatrix}.
\]
The upper-right blocks satisfy
\[
\mathbf H_{lk}=\mathbf H_{kl}^{\top},
\qquad
k\in\{3,4\},\quad l\in\{1,2\}.
\]

The symmetric lifted tensor $\bar{\mathcal B}_{\mathcal A}$ is obtained from
$\mathcal B_{\mathcal A}$ by the symmetrization procedure \eqref{2.3}. Since $\alpha=\beta=1/\sqrt{2}$, we compute the real $Z$-eigenpairs of
$\bar{\mathcal B}_{\mathcal A}$ and recover the corresponding $M$-eigenpairs. The detailed results are summarized in Table \ref{tab1}.

\begin{table}[h]
\centering
\caption{$M$-eigenvalues and corresponding eigenvectors of Example \ref{ex1}.}
\label{tab1}
\renewcommand{\arraystretch}{1.0}
\begin{tabular}{c c l l}
\hline
\textbf{No.} & \textbf{M-eigenvalue} & \textbf{Eigenvector(s)} $\mathbf{x}^\top$ & \textbf{Eigenvector(s)} $\mathbf{y}^\top$ \\
\hline
1 & $-1.2764$ & \makecell[l]{$(-0.8168, 0.5769)$ \\ $(0.5894, 0.8079)$} & \makecell[l]{$(0.5894, 0.8079)$ \\ $(-0.8168, 0.5769)$} \\
2 & $0.0710$  & $(-0.9639, 0.2664)$ & $(-0.9639, 0.2664)$ \\
3 & $0.1242$  & $(-0.5774, 0.8165)$ & $(-0.5774, 0.8165)$ \\
4 & $0.2765$  & \makecell[l]{$(-0.9326, 0.3609)$ \\ $(-0.6402, 0.7682)$} & \makecell[l]{$(-0.6402, 0.7682)$ \\ $(-0.9326, 0.3609)$} \\
\hline
\end{tabular}
\end{table}

As observed from Table \ref{tab1}, Algorithm \ref{alg1} successfully identifies all six real $M$-eigenvalues. Furthermore, our computed eigenvalues are in agreement with the results reported in \cite{QDH09}, which validates the numerical accuracy of the proposed approach. These numerical results also confirm that the theoretical analysis is correct and the proposed algorithm is feasible.

\begin{example}\label{ex2}
Consider the tensor $\mathcal{A} \in \mathbb{P}^{3 \times 3 \times 3 \times 3}$associated with the elastic moduli of trigonal $\mathrm{CaMg(CO_3)_2}$ dolomite reported in \cite{LLL19}. Its independent nonzero entries are given by
\begin{align*}
a_{1111} &= a_{2222} = 196.6,\; a_{2233} = a_{3311} = 83.2,\; a_{1313} = a_{2323} = a_{3131} = a_{3232} = 54.7, \\
a_{2223} &= a_{2232} = -a_{1213} = -a_{2131} = -31.7,\; a_{3333} = 110,\; a_{1212} = a_{2121} = 64.4, \\
a_{1232} &= a_{2321} = -a_{1131} = -a_{1311} = -25.3,\; a_{1321} = a_{3112} = 44.8, \\
a_{1223} &= a_{2132} = -35.84,\; a_{1122} = 132.2,
\end{align*}
and $a_{ijkl} = 0$ otherwise.
\end{example}

To evaluate our approach on this practical application, we employ Algorithm \ref{alg1}. Initially, we construct the lifted tensor $\mathcal{B}_{\mathcal{A}} \in \mathbb{R}^{6\times6\times6\times6}$, which possesses a sparse structure that can be expressed via a $3 \times 3$ block partitioning. Let $\mathbf{0}_3$ be the $3 \times 3$ zero matrix and $\mathbf{H}_{kl} \in \mathbb{R}^{3\times3}$ represent the nonzero sub-blocks. The block representation of the slices $\mathcal{B}_{\mathcal{A}}(:,:,k,l)$ is given as follows

\[
\mathcal{B}_{\mathcal{A}}(:,:,k,l) =
\begin{cases}
    \begin{pmatrix} \mathbf{0}_3 & \mathbf{0}_3 \\ \mathbf{H}_{kl} & \mathbf{0}_3 \end{pmatrix}, & \text{if } l \in \{1,2,3\} \text{ and } k \in \{4,5,6\}, \\[15pt]
    \begin{pmatrix} \mathbf{0}_3 & \mathbf{H}_{kl} \\ \mathbf{0}_3 & \mathbf{0}_3 \end{pmatrix}, & \text{if } l \in \{4,5,6\} \text{ and } k \in \{1,2,3\}, \\[15pt]
    ~~~\mathbf{0}_{6 \times 6}, & \text{otherwise},
\end{cases}
\]
where the non-zero sub-blocks $\mathbf{H}_{kl}$ satisfy $\mathbf{H}_{lk} = \mathbf{H}_{kl}^\top$. Their values for the indices $l \in \{1,2,3\}$ and $k \in \{4,5,6\}$ are listed below.

{\footnotesize 
\begin{align*}
\mathbf{H}_{41} &= \begin{pmatrix}
196.6 & 0 & 25.3 \\
0 & 132.2 & 44.8 \\
25.3 & 44.8 & 83.2
\end{pmatrix}, &
\mathbf{H}_{51} &= \begin{pmatrix}
0 & 132.2  & 44.8  \\
64.4  & 0&  -25.3  \\
31.7  & -35.84  & 0
\end{pmatrix}, &
\mathbf{H}_{61} &= \begin{pmatrix}
25.3 & 44.8  &83.2\\
31.7 &-35.84 &0\\
54.7 &   0   &0
\end{pmatrix}. \\[10pt] 
\mathbf{H}_{42} &= \begin{pmatrix}
0 & 64.4 & 31.7 \\
132.2 & 0 & -35.84 \\
44.8 & -25.3 & 0
\end{pmatrix}, &
\mathbf{H}_{52} &= \begin{pmatrix}
132.2 & 0 & -35.84 \\
0 & 196.6 & -31.7\\
-35.84 & -31.7 & 83.2
\end{pmatrix},&
\mathbf{H}_{62} &= \begin{pmatrix}
44.8 & -25.3 & 0 \\
-35.84 & -31.7 &83.2 \\
0 & 54.7 & 0
\end{pmatrix}. \\[10pt]
\mathbf{H}_{43} &= \begin{pmatrix}
25.3 & 31.7 & 54.7 \\
44.8 & -35.84 & 0 \\
83.2  & 0 & 0
\end{pmatrix}, &
\mathbf{H}_{53} &= \begin{pmatrix}
44.8 & -35.84 & 0 \\
-25.3 & -31.7 &54.7 \\
0 &83.2 &0
\end{pmatrix}, &
\mathbf{H}_{63} &= \begin{pmatrix}
83.2 & 0 & 0 \\
0 & 83.2 & 0 \\
0 &0 & 110
\end{pmatrix}.
\end{align*}
}

For the remaining cases where $l \in \{4, 5, 6\}$ and $k \in \{1, 2, 3\}$, the nonzero sub-blocks are located in the upper-right position and are determined by $\mathbf{H}_{lk} = \mathbf{H}_{kl}^\top$. The detailed results are summarized in Table \ref{tab2}.


\begin{table}[!ht]
    \centering
    \caption{$M$-eigenvalues and corresponding eigenvectors of Example \ref{ex2}.}
    \label{tab2}
    \renewcommand{\arraystretch}{0.7}
    \begin{tabular}{lllll}
    \hline
        \textbf{No.} & \textbf{M-eigenvalue} & ~ \textbf{Eigenvector(s) $\mathbf{x}^\top$}  & ~ \textbf{Eigenvector(s) $\mathbf{y}^\top$}   \\ \hline
        1 & $-37.3664$  &\makecell[l]{$\pm$( -0.7072, -0.7055, -0.0463)\\ $\pm$(-0.6107, 0.5913,  0.5267)}&  \makecell[l]{$\pm$(-0.6107, 0.5913, 0.5267 )\\ $\pm$(-0.7072, -0.7055, -0.0463)} \\
        2 & 14.5704  &\makecell[l]{$\pm$( -0.6208, 0.5692, 0.5391)\\ $\pm$( -0.2833, 0.3346,  -0.8988)}& \makecell[l]{$\pm$( -0.2833, 0.3346, -0.8988)\\$\pm$(-0.6208, 0.5692, 0.5391)} \\
        3& 15.8605  &\makecell[l]{$\pm$(-0.3304, -0.4095, -0.8504)\\ $\pm$(-0.0383, 0.8731,   -0.4859)} &\makecell[l]{$\pm$( -0.0383, 0.8731, -0.4859  )\\$\pm$( -0.3304, -0.4095,   -0.8504)} \\
        4 & 18.3959  &\makecell[l]{$\pm$( -0.8957, 0.0044, -0.4447)\\ $\pm$(-0.3516, -0.3897,   0.8512)}&\makecell[l]{ $\pm$( -0.3516, -0.3897, 0.8512)\\$\pm$( -0.8957, 0.0044, -0.4447  )} \\
        5 & 23.9870  & \makecell[l]{$\pm$( -0.8455, -0.1045, -0.5237)\\ $\pm$( -0.0671, -0.8071,   0.5866)} &\makecell[l]{$\pm$(-0.0671, -0.8071, 0.5866)\\$\pm$( -0.8455, -0.1045, -0.5237 )} \\
        6 & 59.1027  &\makecell[l]{ $\pm$(-0.1834, -0.9311, -0.3152)\\$\pm$(-0.0361, -0.0931,   0.9950)}&\makecell[l]{ $\pm$( -0.0361, -0.0931, 0.9950)\\$\pm$( -0.1834, -0.9311,-0.3152  )} \\
        7& 59.4996  &\makecell[l]{ $\pm$( -0.9114, -0.2454, 0.3304)\\$\pm$(-0.0932, -0.0478,   -0.9945)} &\makecell[l]{$\pm$(-0.0932, -0.0478, -0.9945)\\$\pm$( -0.9114, -0.2454, 0.3304  )} \\
        8 & 62.2588  &\makecell[l]{ $\pm$( -0.6709, 0.6157, -0.4133)\\$\pm$(-0.0880, 0.0869,   0.9923)}&\makecell[l]{$\pm$(-0.0880, 0.0869, 0.9923)\\$\pm$(-0.6709, 0.6157, -0.4133)} \\
        9 & 69.9094 & \makecell[l]{$\pm$( -0.8863, -0.3057, 0.3478)\\ $\pm$(-0.5150, 0.6416,   -0.5684)} &\makecell[l]{ $\pm$(-0.5150, 0.6416, -0.5684)\\$\pm$(-0.8863, -0.3057, 0.3478  )} \\
        10 & 73.1065  &\makecell[l]{ $\pm$(-0.6885, 0.4545, -0.5651)\\$\pm$(-0.2551, -0.9180,   -0.3038)} & \makecell[l]{$\pm$( -0.2551, -0.9180, -0.3038)\\$\pm$(-0.6885, 0.4545, -0.5651  )} \\
        11 & 88.8932 & \makecell[l]{$\pm$(-0.9224, -0.1251, 0.3654)\\$\pm$(-0.0650, -0.9470,   -0.3146)} & \makecell[l]{$\pm$(-0.0650, -0.9470, -0.3146) \\$\pm$( -0.9224, -0.1251,   0.3654)} \\
        12 & 90.4814  &\makecell[l]{$\pm$(-0.6712, -0.7411, 0.0160)\\ $\pm$(-0.3926, 0.3561,   -0.8480)} &\makecell[l] {$\pm$(-0.3926, 0.3561, -0.8480) \\ $\pm$( -0.6712, -0.7411, 0.0160  )} \\
        13 & 97.1101  &\makecell[l]{$\pm$( -0.6743, -0.1634, 0.7201)\\ $\pm$( -0.2116, 0.9754,   -0.0624)}&\makecell[l]{$\pm$( -0.2116, 0.9754, -0.0624)\\ $\pm$(-0.6743, -0.1634, 0.7201  )} \\
        14 & 98.4839 &\makecell[l]{$\pm$( -0.9872, 0.1525, -0.0472)\\$\pm$( -0.1248, -0.7070,   -0.6961)}&\makecell[l]{ $\pm$( -0.1248, -0.7070, -0.6961)\\$\pm$( -0.9872, 0.1525, -0.0472  )} \\
        15 & 110.0000  & $\pm$( 0.0000   0.0000   -1.0000  ) & $\pm$( 0.0000   0.0000   -1.0000  ) \\
        16 & 205.5600  & $\pm$( -0.9698   0.1913   -0.1513  ) & $\pm$( -0.9698   0.1913   -0.1513  ) \\
        17 & 211.9361  & $\pm$( -0.2466   0.9513   -0.1851  ) & $\pm$( -0.2466   0.9513   -0.1851  ) \\
        18 & 261.9312  & $\pm$( -0.6593   -0.7494   0.0604  ) & $\pm$( -0.6593   -0.7494   0.0604  ) \\
        19 & 272.3828  & $\pm$( -0.3735   -0.8344   0.4054  ) & $\pm$( -0.3735   -0.8344   0.4054  ) \\
        20 & 281.8873  & $\pm$( -0.8111   -0.4324   -0.3939  ) & $\pm$( -0.8111   -0.4324   -0.3939  ) \\
        21 & 318.0790  & $\pm$( -0.6638   0.6374   0.3913  ) & $\pm$( -0.6638   0.6374   0.3913  ) \\ \hline
    \end{tabular}
\end{table}

Table~\ref{tab2} lists the 21 distinct nonzero real $M$-eigenvalues obtained by the proposed method. In addition, a recoverable zero eigenpair with two nonzero vector blocks was found, showing that zero is also an $M$-eigenvalue of $\mathcal A$. Thus, the computation returns 22 distinct real
$M$-eigenvalues in total. However, $\lambda = 0$ possesses a large number of corresponding eigenvectors. Therefore, for the sake of brevity, the zero eigenvalue and its associated eigenvectors are excluded from Table \ref{tab2}, which exclusively presents the 21 nonzero $M$-eigenpairs.

The largest computed $M$-eigenvalue is $318.0790$, in agreement with the reference result in \cite{LLL19}. The smallest $M$-eigenvalue is $-37.3664$, and therefore the associated elasticity tensor fails the strong ellipticity condition.

\begin{example}\label{ex3}
Consider the tensor $\mathcal{A}\in\mathbb{R}^{4\times 3\times4\times3}$ whose entries are randomly generated from the standard normal distribution $ N(0,1)$.
\end{example}

This example is used to investigate the numerical effect of the normalization parameters $\alpha$ and $\beta$. In particular, we compare the standard $Z$-eigenvalue formulation with the proposed
$\mathbf B_{\alpha,\beta}$-eigenvalue formulation for several representative positive choices of $(\alpha,\beta)$. The standard $Z$-eigenvalue formulation corresponds to
\[
\alpha=\beta=\frac{1}{\sqrt{2}},
\qquad
\mathbf B_{\alpha,\beta}=\mathbf I_{7}.
\]

For a consistent comparison, all parameter choices, including the standard $Z$ case, are treated using the generalized tensor eigenvalue method developed in \cite{CHZ16}. For each
$(\alpha,\beta)$, we compute the real $\mathbf B_{\alpha,\beta}$-eigenpairs of
$\bar{\mathcal B}_{\mathcal A}$ and recover the corresponding $M$-eigenpairs.

For each case, the computed generalized eigenpairs are recovered as $M$-eigenpairs, and the residuals are evaluated by \eqref{4.1}. In Table~\ref{tab3}, NE denotes the number of computed distinct real $M$-eigenvalues, CPU denotes the computational time in seconds,
$\mathrm{Res}_{\max}$ denotes the maximum residual, and $\mathrm{Res}_{\mathrm{mean}}$ denotes the mean residual.
%


\begin{table}[htbp]
\centering
\caption{Comparison of different normalization parameters
$\alpha$ and $\beta$ in Example \ref{ex3}.}
\label{tab3}
\begin{tabular}{cccccc}
\hline
$\alpha$ & $\beta$ & NE & CPU(s) & $\mathrm{Res}_{\max}$ & $\mathrm{Res}_{\mathrm{mean}}$ \\
\hline
 $\frac{1}{\sqrt{2}}$ & $\frac{1}{\sqrt{2}}$ & 37 & 36.82 & $1.01\times 10^{-15}$ & $5.83\times 10^{-16}$ \\
 0.5 & 1.0 & 37 & 35.67 & $1.49\times 10^{-15}$ & $6.13\times 10^{-16}$ \\
 1.0 & 0.5 & 37 & 35.22 & $1.20\times 10^{-15}$ & $6.19\times 10^{-16}$ \\
 0.5 & 0.5 & 37 & 35.01 & $3.06\times 10^{-15}$ & $5.80\times 10^{-16}$ \\
 1.0 & 1.0 & 37 & 35.02 & $1.14\times 10^{-15}$ & $6.15\times 10^{-16}$ \\
 1.5 & 0.75 & 37 & 35.36 & $1.34\times 10^{-15}$ & $6.51\times 10^{-16}$ \\
\hline
\end{tabular}
\end{table}

From Table~\ref{tab3}, all tested choices of $(\alpha,\beta)$ recover 37 distinct real $M$-eigenvalues. The maximum and mean residuals remain at the level of $10^{-15}$, which confirms the numerical accuracy of the computed $M$-eigenpairs. Moreover,
several choices of $(\alpha,\beta)$ require slightly less CPU time than the standard $Z$ case. This indicates that the parameterized $\mathbf B_{\alpha,\beta}$ formulation provides additional flexibility in practical computations without
loss of accuracy.

\begin{example}\label{ex4}
Consider the tensor $\mathcal{A}\in\mathbb{R}^{m\times n\times m\times n}$ whose entries are uniformly generated in $(0,1)$, i.e., $\mathcal{A}\sim U(0,1)$.
\end{example}

To examine the scalability of the proposed formulation, we generate random fourth order partially symmetric tensors with different dimension pairs $(m,n)$. For each dimension pair, 50 independent random tensors are generated.  The independent entries are sampled from the uniform
distribution, and the remaining entries are obtained according to the hierarchical symmetry relations.

For a fair comparison, the same random tensors are used for the standard $Z$ formulation and the
$\mathbf B_{\alpha,\beta}$ formulation. The standard $Z$ formulation corresponds to the
special choice $\alpha=\beta=1/\sqrt{2}$. For the $\mathbf B_{\alpha,\beta}$ formulation, several
representative parameter choices are tested. The parameters are fixed in advance and kept unchanged for all random instances of the same dimension.

For each case, the computed generalized eigenpairs are recovered as $M$-eigenpairs,
and the residuals are evaluated by \eqref{4.1}. In Table~\ref{tab4}, NE denotes the average number of computed distinct real $M$-eigenvalues, $\overline{\mathrm{CPU}}_Z$ and $\overline{\mathrm{CPU}}_B$ denote the average CPU times over 50 trials, and $\mathrm Res_Z$ and $\mathrm Res_B$ denote the corresponding mean residuals. The CPU ratio is defined by
\[
\mathrm{Ratio}
=
\frac{\overline{\mathrm{CPU}}_Z}
{\overline{\mathrm{CPU}}_B}.
\]
A value of $\mathrm{Ratio}>1$ indicates that the $\mathbf B_{\alpha,\beta}$ formulation requires less average CPU time than the standard $Z$ formulation.
\begin{table}[htbp]
\centering
\caption{Comparison of the standard $Z$ formulation and the
$\mathbf B_{\alpha,\beta}$ formulation across varying dimensions in Example 4.4.}
\label{tab4}
\renewcommand{\arraystretch}{1.05}
\resizebox{\textwidth}{!}{
\begin{tabular}{cccccccccccc}
\hline
$m$ & $n$ & $m+n$ & $\alpha$ & $\beta$
& $\mathrm{NE}_Z$ & $\mathrm{NE}_B$
& $\overline{\mathrm{CPU}}_Z$ & $\overline{\mathrm{CPU}}_B$
& Ratio & $\mathrm{Res}_Z$ & $\mathrm{Res}_B$ \\
\hline
2 & 2 & 4 & 0.4 & 0.8
& 7 & 7 & 0.791 & 0.263 & 3.006
& $5.04\times10^{-16}$ & $4.16\times10^{-16}$ \\
2 & 3 & 5 & 0.2 & 0.4
& 13 & 13 & 1.428 & 1.238 & 1.153
& $5.11\times10^{-16}$ & $4.54\times10^{-16}$ \\
2 & 4 & 6 & 0.2 & 1.0
& 21 & 21 & 6.356 & 5.671 & 1.121
& $4.67\times10^{-16}$ & $4.40\times10^{-16}$ \\
3 & 2 & 5 & 0.8 & 0.8
& 13 & 13 & 1.396 & 1.292 & 1.080
& $4.57\times10^{-16}$ & $4.11\times10^{-16}$ \\
3 & 3 & 6 & 0.6 & 1.0
& 20 & 20 & 7.777 & 7.182 & 1.083
& $6.30\times10^{-16}$ & $5.71\times10^{-16}$ \\
3 & 4 & 7 & 0.6 & 1.0
& 23 & 23 & 36.991 & 31.086 & 1.190
& $5.76\times10^{-16}$ & $4.88\times10^{-16}$ \\
4 & 2 & 6 & 1.0 & 0.4
& 19 & 19 & 6.335 & 5.182 & 1.223
& $5.44\times10^{-16}$ & $5.14\times10^{-16}$ \\
4 & 3 & 7 & 0.6 & 0.6
& 45 & 45 & 33.330 & 30.704 & 1.086
& $5.79\times10^{-16}$ & $4.99\times10^{-16}$ \\
\hline
\end{tabular}
}
\end{table}

From Table~\ref{tab4}, both formulations compute the same number of distinct real $M$-eigenpairs for all tested dimension pairs. The mean residuals remain at the level of $10^{-16}$, which confirms the accuracy of the computed $M$-eigenpairs. As expected, the computational cost increases with the tensor dimension due to the increasing complexity of the associated polynomial eigenvalue problem. For the tested instances, the selected $\mathbf B_{\alpha,\beta}$ formulations achieve ratio factors larger than one, indicating lower average CPU times compared with the standard $Z$ formulation.

To better illustrate the CPU comparison, the ratio in Table~\ref{tab4} are plotted in Figure~\ref{fig1}. The dashed line indicates the unit ratio.

\begin{figure}[h]
        \centering
       \includegraphics[width=0.6\textwidth]{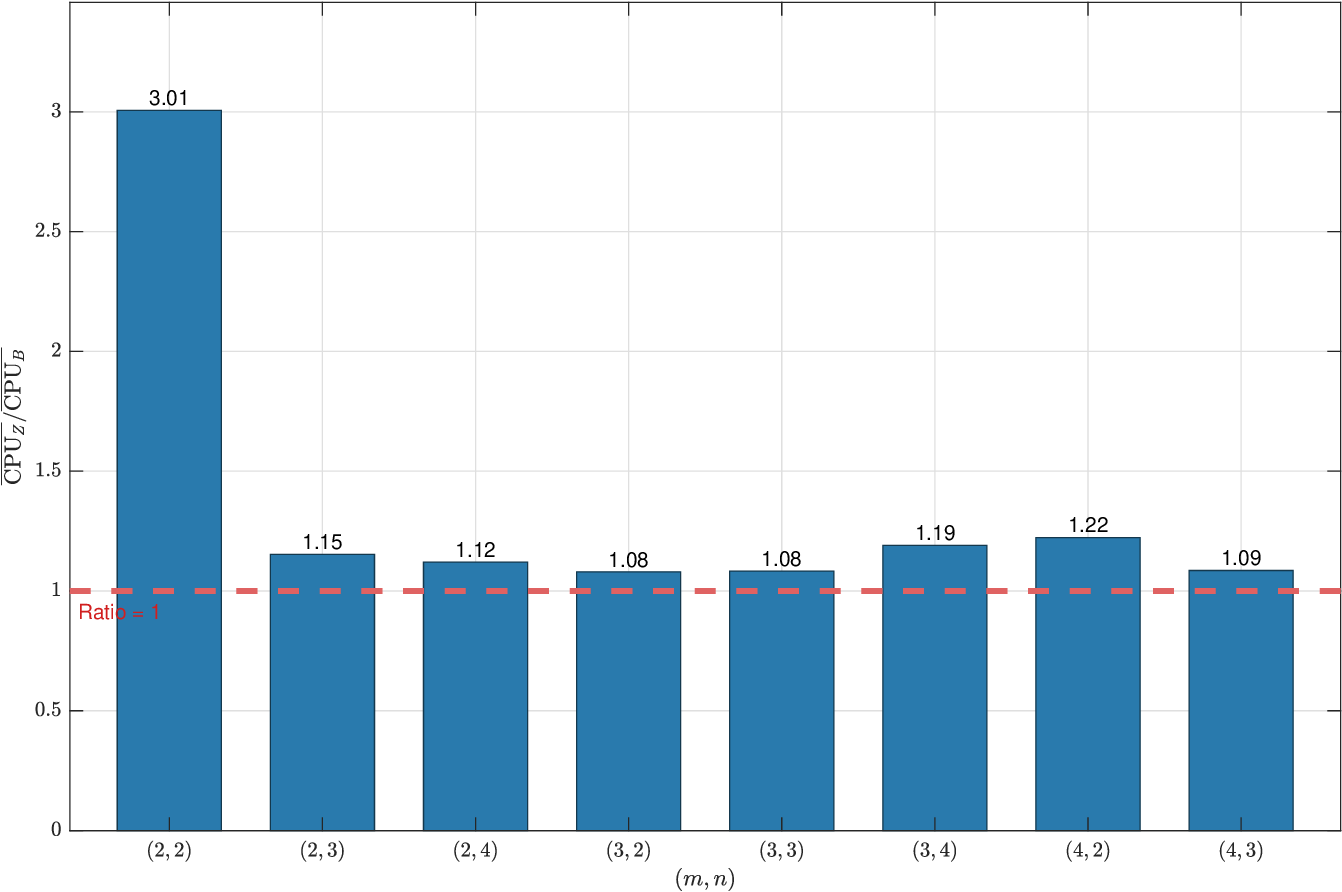}
        \caption{CPU time ratio between the standard $Z$ formulation and the
$\mathbf B_{\alpha,\beta}$ formulation in Example 4.4.}
        \label{fig1}
\end{figure}

As shown in Figure~\ref{fig1}, all CPU ratios are above the unit line. This indicates that the $B_{\alpha,\beta}$ formulation with the selected parameters is faster than the standard $Z$ formulation for the tested random instances. Together with the residuals reported in
Table~\ref{tab4}, these results indicate that the parameterized formulation can reduce computational time without loss of numerical accuracy.

\begin{example}\label{ex5}
To illustrate the perturbation bounds established in Theorem \ref{th3.1}, we use the tensor in Example \ref{ex2} as the unperturbed tensor $\mathcal A$. Let
$\mathcal E=\varepsilon\mathcal E_0,\widetilde{\mathcal A}=\mathcal A+\mathcal E$,
where $\mathcal E_0$ is a hierarchically symmetric tensor normalized such that $\|\mathcal{E}_0\|_F = 1$. The perturbation level is chosen from
\[
\varepsilon\in\{1,10^{-1},10^{-2},10^{-3},10^{-4},10^{-5}\}.
\]
For each $\varepsilon$, we compute the perturbed largest
$M$-eigenvalue $\lambda_{M_{\max}}(\widetilde{\mathcal A})$ and the corresponding lower and upper bounds
$$
L_M=\lambda_{M\max}(\mathcal A)+\lambda_{M\min}(\mathcal E),
\qquad
U_M=\lambda_{M\max}(\mathcal A)+\lambda_{M\max}(\mathcal E).
$$
We further report
$$
\mathrm{Err}=\left|\lambda_{M\max}(\widetilde{\mathcal A})
-\lambda_{M\max}(\mathcal A)\right|,
\qquad
\mathrm{Width}=U_M-L_M.
$$
The numerical results are presented in Table~\ref{tab5}.
\end{example}

\begin{table}[htbp]
  \centering
  \caption{Perturbation bounds for $\lambda_{M\max}(\widetilde{\mathcal A})$ under different perturbation levels.}
  \label{tab5}
  \begin{tabular}{cccccc}
    \toprule
    $\mathrm \varepsilon$ & $\mathrm L_M$ & $\mathrm\lambda_{M\max}(\widetilde{\mathcal A})$ & $\mathrm U_M$ &$\mathrm Err$ & $\mathrm Width$ \\
    \midrule
    $1$        & $317.70270019$ & $318.27823143$ & $318.66275604$ &$0.19921796$ &$0.96005585$ \\
    $10^{-1}$  & $318.04138214$ & $318.09888719$ & $318.13738773$ &$0.01987371$ &$0.09600558$\\
    $10^{-2}$  & $318.07525034$ & $318.08100036$ & $318.08485090$ &$0.00198689$ &$0.00960056$ \\
    $10^{-3}$  & $318.07863716$ & $318.07921216$ & $318.07959721$ &$0.00019868$ &$0.00096006$ \\
    $10^{-4}$  & $318.07897584$ & $318.07903333$ & $318.07907185$ &$0.00001987$ &$0.00009601$\\
    $10^{-5}$  & $318.07900971$ & $318.07901546$ & $318.07901931$ &$0.00000199$ &$0.00000960$\\
    \bottomrule
  \end{tabular}
\end{table}

As shown in Table~\ref{tab5}, the computed value $\lambda_{M\max}(\widetilde{\mathcal A})$ always lies in the interval $[L_M,U_M]$, which confirms the perturbation bound in Theorem~\ref{th3.1}.
Moreover, as $\varepsilon$ decreases, both bounds become tighter and the interval
width $U_M-L_M$ decreases rapidly, indicating that the estimate becomes
increasingly accurate for smaller perturbations.

\section{Conclusions}

In this paper, we developed a parameterized generalized tensor eigenvalue reformulation for computing all real $M$-eigenpairs of fourth order partially symmetric tensors. By constructing an associated symmetric lifted tensor and introducing the weighted normalization matrix $\mathbf B_{\alpha,\beta}$, the $M$-eigenvalue problem was reformulated as a $\mathbf B_{\alpha,\beta}$ eigenvalue problem. An exact eigenpair correspondence and explicit recovery formulas were
established, and the standard $Z$-eigenvalue formulation was recovered as a special case. We further derived perturbation bounds for the largest $M$-eigenvalue and showed that the lifting introduces no additional relaxation in these bounds. Numerical results were
reported to demonstrate the effectiveness and flexibility of the proposed framework.

\medskip



\begin{thebibliography}{99}

\bibitem{AG91} Auchmuty, G.: Globally and rapidly convergent algorithms for symmetric eigenproblems. SIAM Journal on Matrix Analysis and Applications. 12(4): 690-706 (1991)
\bibitem{B1970} Backus, G.: A geometrical picture of anisotropic elastic tensors. Reviews of Geophysics. 8: 633-671 (1970)
\bibitem{CCQ16} Chang, J., Chen, Y., Qi, L.: Computing eigenvalues of large scale sparse tensors arising from a hypergraph. SIAM Journal on Scientific Computing. 38(6): A3618-A3643 (2016)
\bibitem{CP09} Chang, K., Pearson, K., Zhang, T.: On eigenvalue problems of real symmetric tensors. Journal of Mathematical Analysis and Applications. 350: 416-422 (2009)
\bibitem{CCW19} Che, H., Chen, H., Wang, Y.: M-positive semi-definiteness and M-positive definiteness of fourth-order partially symmetric Cauchy tensors. Journal of Inequalities and Applications. 32 (2019)
\bibitem{CCW20} Che, H., Chen, H., Wang, Y.: On the M-eigenvalue estimation of fourth-order partially symmetric tensors. Journal of Industrial and Management Optimization. 16(1): 309-324 (2020)
\bibitem{CCX20} Che, H., Chen, H., Xu, N., Zhu, Q.: New lower bounds for the minimum M-eigenvalue of elasticity M-tensors and applications. Journal of Inequalities and Applications. 2020, 188: (2020)
\bibitem{CCZ21} Che, H., Chen, H., Zhou, G.: New M-eigenvalue intervals and application to the strong ellipticity of fourth-order partially symmetric tensors. Journal of Industrial and Management Optimization. 17(6): 3685-3694 (2021)
\bibitem{CH22} Chen, H., He, H., Wang, Y., Zhou, G.: An efficient alternating minimization method for fourth degree polynomial optimization. Journal of Global Optimization. 82(1): 83-103 (2022)
\bibitem{CHZ16} Chen, L., Han, L., Zhou, L.: Computing tensor eigenvalues via homotopy methods. SIAM Journal on Matrix Analysis and Applications. 37(1): 290-319 (2016)
\bibitem{CH22} Cui, L., Hu, Q., Chen, Y., Song, Y.: A Rayleigh quotient-gradient neural network method for computing Z-eigenpairs of general tensors. Numerical Linear Algebra with Applications. 29(3): e2420 (2022)
\bibitem{DS26} Du, Z., Song, Y.: An Efficient Memory Gradient Method for Extreme M-Eigenvalues of Elastic type Tensors. Computational and Applied Mathematics. 45: 430 https://doi.org/10.1007/s40314-026-03807-0 (2026)
\bibitem{DW24} Du, Z., Wang, C., Chen, H., Yan, H.: An Efficient GIPM Algorithm for Computing the Smallest V-Singular Value of the Partially Symmetric Tensor. Journal of Optimization Theory and Applications. 201(3): 1151-1167 (2024)
\bibitem{HLW20} He, J., Li, C., Wei, Y.: M-eigenvalue intervals and checkable sufficient conditions for the strong ellipticity. Applied Mathematics Letters. 102: 106137 (2020)
\bibitem{HLX21} He, J., Liu, Y., Xu, G.: New S-type inclusion theorems for the M-eigenvalues of a 4th-order partially symmetric tensor with applications. Applied Mathematics and Computation. 398: 125992 (2021)
\bibitem{HL21} He, J., Liu, Y., Xu, G.: New M-eigenvalue inclusion sets for fourth-order partially symmetric tensors with applications. Bulletin of the Malaysian Mathematical Sciences Society. 44(6): 3929-3947 (2021)
\bibitem{HXL20} He, J., Xu, G., Liu, Y.: Some inequalities for the minimum M-eigenvalue of elasticity M-tensors. Journal of Industrial and Management Optimization. 16(6): 3035-3045 (2020)
\bibitem{KB09} Kolda, T., Bader, B.: Tensor decompositions and applications. SIAM review, 51(3): 455-500 (2009)
\bibitem{LCL22} Li,S., Chen,Z., Liu,Q., Lu,L.: Bounds of M-eigenvalues and strong ellipticity conditions for elasticity tensors. Linear Multilinear Algebra. 70(19): 4544-4557 (2022)
\bibitem{LLL19} Li, S., Li, C., Li, Y.: M-eigenvalue inclusion intervals for a fourth order partially symmetric tensor. Journal of Computational and Applied Mathematics. 356: 391-401 (2019)
\bibitem{LL21} Li, S., Li, Y.: Programmable sufficient conditions for the strong ellipticity of partially symmetric tensors. Applied Mathematics and Computation. 403: 126134 (2021)
\bibitem{Lim05} Lim, L.: Singular values and eigenvalues of tensors: a variational approach[C]//1st IEEE International Workshop on Computational Advances in Multi-Sensor Adaptive Processing. IEEE. 2005: 129-132 (2005)
\bibitem{LZ16} Liu, L., Zhou, G., Caccetta, L.: Fixed point methods for computing a Z-eigenpair of general square tensors. Pacific Journal of Optimization. 12(2): 367-378 (2016)
\bibitem{MW20} Mo, C., Wang, X., Wei, Y.: Time-varying generalized tensor eigenanalysis via Zhang neural networks. Neurocomputing. 407: 465-479 (2020)
\bibitem{PA24} Pakmanesh, M., Afshin, H., Hajarian, M.: Normalized Newton method to solve generalized tensor eigenvalue problems. Numerical Linear Algebra with Applications. 31(4): e2547 (2024)
\bibitem{Qi05} Qi L. Eigenvalues of a real supersymmetric tensor. Journal of symbolic computation. 40(6): 1302-1324 (2005)
\bibitem{QC18} Qi, L., Chen, H., Chen, Y.: Tensor Eigenvalues and Their Applications. Springer, Singapore (2018)
\bibitem{QDH09} Qi, L., Dai, H.  Han, D.: Conditions for strong ellipticity and M-eigenvalues. Frontiers of Mathematics in China. 4: 349-364 (2009)
\bibitem{WCW23} Wang, C., Chen, H., Wang, Y.: An alternating shifted inverse power method for the extremal eigenvalues of fourth order partially symmetric tensors. Applied Mathematics Letters. 141: 108601 (2023)
\bibitem{WQZ09} Wang, Y., Qi, L., Zhang, X.: A practical method for computing the largest M-eigenvalue of a fourth order partially symmetric tensor. Numerical Linear Algebra with Applications. 16(7): 589-601 (2009)
\bibitem{Zhao23} Zhao, J.: Conditions of strong ellipticity and calculations of M-eigenvalues for a partially symmetric tensor. Applied Mathematics and Computation. 458: 128245 (2023)
\bibitem{ZL24} Zhao, J., Liu, P., Sang, C.: Shifted inverse power method for computing the smallest M-eigenvalue of a fourth-order partially symmetric tensor. Journal of Optimization Theory and Applications. 200(3): 1131-1159 (2024)

\end{thebibliography}
\end{document}